\documentclass[preprint,12pt]{elsarticle}

\usepackage{amssymb}
\usepackage{amsmath}
 \usepackage{amsthm}

\usepackage[disable]{todonotes}
\usepackage[ margin=1.15in]{geometry}

\usepackage[pdftex,colorlinks=true,linkcolor=darkblue,citecolor=darkred,urlcolor=blue]{hyperref}

\newtheorem{theorem}{Theorem}
\theoremstyle{definition}
\newtheorem{definition}{Definition}
\newtheorem{proposition}{Proposition}
\newtheorem{lemma}{Lemma}
\newtheorem{remark}{Remark}
\newtheorem{hypothesis}{Hypothesis}
\newtheorem{corollary}{Corollary}
\newtheorem{conjecture}{Conjecture}
 
 \newcommand{\id}{\mathrm I}
 \newcommand{\mB}{\mathbf{B}}
 \newcommand{\mR}{\mathbf{R}}
 \newcommand{\mF}{\mathbf{F}}
 \newcommand{\mG}{\mathbf{G}}
 \newcommand{\mK}{\mathbf{K}}
 \newcommand{\mP}{\mathbf{P}}
 \newcommand{\R}{\mathbb{R}}
 \newcommand{\mM}{\mathbf{M}}
 \newcommand{\mT}{\mathbf{T}}
 \newcommand{\mQ}{\mathbf{Q}}
 \newcommand{\mU}{\mathbf{U}}
 \newcommand{\mW}{\mathbf{W}}
 \newcommand{\mA}{\mathbf{A}}

 \newcommand{\mfu}{\mathbf{u}}
 \newcommand{\mw}{\mathbf{w}}
 \newcommand{\mv}{\mathbf{v}}
 \newcommand{\sfv}{\mathsf{v}}
 \newcommand{\mr}{\mathbf{r}}
 \newcommand{\mh}{\mathbf{h}}
 
 \newcommand{\norm}[1]{\left \| #1 \right \|}

 \newcommand{\Lp}{\mathrm{L}}
 \newcommand{\Hp}{\mathrm{H}}
  \newcommand{\Wp}{\mathrm{W}^{1,1}}

 \newcommand{\rem}[1]{\todo[inline, color=gray!60!white]{\small \texttt{Remarque}: #1}}

\graphicspath{{Figures/}}

\begin{document}
\begin{frontmatter}



\title{Theoretical / numerical study of modulated traveling waves\\ in inhibition stabilized networks}

 \author{Safaa Habib}
% \affiliation[label1, label2]{organization=Inria Center of University Côte d'Azur,
%%             country={}
%             }
%%

\author[1]{Romain Veltz} %% Author name

 \affiliation[1]{organization={Inria Center at Université Côte d'Azur, Cronos team},
%	addressline={},
%	city={},
%	postcode={},
%	state={},
	country=France
}
\cortext[1]{Corresponding author}

%% Abstract
\begin{abstract}
We prove a principle of linearized stability for traveling wave solutions to neural field equations posed on the real line. Additionally, we provide the existence of a finite dimensional invariant center manifold close to a traveling wave, this allows to study bifurcations of traveling waves. Finally, the spectral properties of the modulated traveling waves are investigated. Numerical schemes for the computation of modulated traveling waves are provided. We then apply these results and methods to study a neural field model in a inhibitory stabilized regime. We showcase Fold, Hopf and Bodgdanov-Takens bifurcations of traveling pulses. Additionally, we continue the modulated traveling pulses as function of the time scale ratio of the two neural populations and show numerical evidences for snaking of modulated traveling pulses.
\end{abstract}

%Graphical abstract
\begin{graphicalabstract}
	\begin{figure*}
\includegraphics{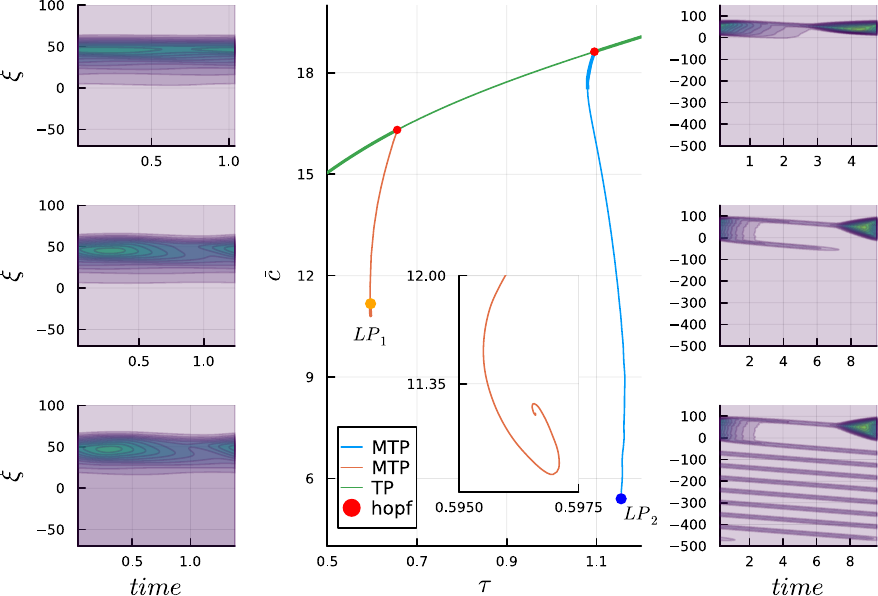}
\caption{Middle: curves of traveling pulses (TP) and modulated traveling pulses (MTP) for $\theta_i = 0.38$. Inset: zoom on the left MTP branch.  Left: example of MTPs on the left branch in the wave frame. Right: example of MTPs on the right branch in the wave frame. The MTP branches were computed with a single standard shooting \eqref{eq:SHODE}. For the branch of MTPs, we plot the value of the speed $\bar c$ as expressed in \eqref{eq:SHODE}. The stable states are marked with a thick line and the unstable ones with a thin line. The points $LP_1, LP_2$ marked with orange/blue dots represents limits points at which the continuation stopped.}
\end{figure*}
\end{graphicalabstract}

%Research highlights
\begin{highlights}
\item Principle of linearized stability for traveling wave solutions to neural field equations.
\item Bifurcation theory for traveling wave solutions.
\item Modulated traveling waves are orbitally linearly stables.
\item Numerical Methods for computing modulated traveling waves.
\item Numerical evidence for snaking of modulated traveling waves.
\end{highlights}

%% Keywords

\end{frontmatter}



\section{Introduction}\label{sec:intro}
Cortical waves have been the subject of active studies as reviewed in  \cite{muller_cortical_2018}. Some of the latest results concern the stimulus evoked activity in the visual cortex of awake monkeys \cite{muller_stimulus-evoked_2014}  where it was experimentally shown that this cortical activity is a traveling wave (see Figure~\ref{fig0}) thanks to a recent phase-based algorithm for wave detection. It was suggested in \cite{muller_cortical_2018} that the main mechanism for the existence of this traveling wave is a non zero - distance dependent conduction - delay.

On the other hand, the visual cortex in cats was shown in \cite{ozeki_inhibitory_2009} to be best modeled with an inhibition stabilized network (ISN) in which a strong recurrent inhibitory input onto the excitatory cells paradoxically shapes the dynamics \cite{tsodyks_paradoxical_1997}. The dynamics of traveling pulses and fronts in this type of network was very nicely studied in \cite{harris_traveling_2018} as function of the inhibition space / time scales. Notably, the authors found modulated traveling pulses by direct simulation. They also narrowed the analysis to the specific exponential connectivity kernels in order to perform numerical continuation of the waves by relying on the spatial dynamics trick \cite{sandstede_stability_2002}. 

Here, we theoretically and numerically investigate the traveling wave dynamics in ISNs in hope to elucidate key mechanisms in monkey visual cortex dynamics \cite{muller_stimulus-evoked_2014}. We thus do not impose the connectivity kernel function.
We do not study the existence of wave \cite{ermentrout_existence_1993, faye_traveling_2018, faye_propagation_2019} and simply assume that one is provided. At Hopf bifurcations in the wave frame, modulated traveling waves (MTW) emerge which are the main subject of this study. We also aim to study how the principle of linearized stability applies to TW and MTW in the context of neural fields.  This was for example conjectured in \cite{faye_traveling_2018} although this result is well known in the context of reaction-diffusion models \cite{sandstede_structure_2001, sandstede_stability_2002}. We based our analysis on the freezing method \cite{beyn_freezing_2004, rottmann-matthes_linear_2011}. 
When used numerically, this method allows us to simulate the long time asymptotic of waves without being limited by the simulation domain as in \cite{harris_traveling_2018}.

The (numerical) study of waves has a long history in the context of neural fields as reviewed in \cite{coombes_waves_2005,coombes_neural_2014, bressloff_waves_2014},  starting with analytical expressions \cite{amari_dynamics_1977} in which the authors assumed a Heaviside nonlinearity, and spanning the theory as well as the numerical analysis. The proof of existence of waves is a fascinating albeit very technical domain of research \cite{ermentrout_existence_1993} with some of the latest tools reviewed in \cite{faye_propagation_2019}.

The study \cite{rankin_continuation_2014} is among the first one to perform numerical continuation of stationary solutions (not waves) of neural fields equations based on a matrix-free approach, freeing the authors from the use of specific connectivity kernels and space discretization size. This was further extended in \cite{faugeras_spatial_2022} where the numerical bifurcation analysis was entirely performed on GPU. The computation of (stationary) waves using matrix-free methods has first been tackled in \cite{laing_numerical_2014} albeit without the preconditioner discovered in this paper which greatly speeds up the convergence of the newton algorithm.

The literature on the numerical computation of MTW is more scarce in the context of neural fields (but see \cite{wasylenko_bifurcations_2010}) and this work is probably the first study in this area. Their computation as zeros of functionals \cite{laing_numerical_2014}, and not by direct simulations has been performed, for other models, using a shooting method in \cite{wasylenko_bifurcations_2010,garcia_continuation_2016} or a trapezoidal scheme in \cite{uecker_numerical_2021,knobloch_origin_2021,lin_continuation_2018}. See also \cite{sherratt_numerical_2012} for a software dedicated to the computation of MTW in the case of PDEs. In the neural fields community, MTWs have been computed by direct simulation in \cite{blomquist_localized_2005,folias_breathing_2004,folias_traveling_2017,kilpatrick_pulse_2014} or by interface dynamics \cite{coombes_pulsating_2011} but so far, they have not been studied using numerical bifurcation tools. Compared to PDEs where sparse linear algebra can be used, the linear operators are dense which imposes the use of matrix free methods and iterative solvers. These methods are usually impeded by the slow convergence rate of the solvers but we will see that a simple preconditioner can remediate that.

The plan of the paper is as follows. 
In a first section~\ref{section:nf}, we present the neural fields model.
In section~\ref{sec:sumup}, we informally review the concepts associated with the dynamics of waves. 
In section~\ref{section:mainR}, we state our main results.
In section~\ref{section:nlstat}, we prove a principle of linearized stability of traveling wave. In section~\ref{section:cm}, we prove the existence of a center manifold which allows the use of normal form theory and bifurcation analysis. 
In section~\ref{section:mtp-nlstab}, we prove a principle of linearized stability for modulated traveling waves. 
In section~\ref{section:schemesmtp}, we present the different numerical schemes used to compute MTWs before giving a numerical application to a spatialized ISN in section~\ref{section:simus}. We then discuss our results in section~\ref{section:discussion}.

\section{The neural field model}\label{section:nf}
We study a two populations neural field model \cite{bressloff_spatiotemporal_2012, coombes_neural_2014, harris_traveling_2018, veltz_localglobal_2010} which describes the mean neural activity of excitatory and inhibitory populations. We focus on the case where  the neural populations are located on a one-dimensional cortex $\R$ with convolutional connectivity kernels
\begin{equation}\label{eq:nfe}
	\begin{aligned}
		\frac{d u}{d t}(x, t)=&-u(x, t)+S\left(a_{e e} K_{e} \cdot u(x, t)-a_{e i} K_{i}\cdot v(x, t)-\theta_{e}\right)\\
		\tau \frac{d v}{d t}(x, t)=&-v(x, t)+S\left(a_{ie} K_{e} \cdot u(x, t)-a_{i i} K_{i}\cdot v(x, t)-\theta_{i}\right)
	\end{aligned}
\end{equation}
where
\[
K_k\cdot u(x,t):=\int_\R K_k(x-y)u(y,t)dy,\ k\in\{e,i\},\ \quad S(I) := \frac{1}{1+\exp(-\beta I)}.
\]

The parameter $\tau$ describes the ratio of the  inhibitory time constant by the excitatory one. $S$ is a nonlinear function which represents the firing rate as function of the synaptic input and $\beta\geq 0$ is the nonlinear gain. The kernel mappings $K_k$ from $\mathbb R$ to $\mathbb R^+$ represent the spatial spread of the connections between the populations and are normalized to one:
\[
\int_\R K_k=1.
\]
The constants $a_{jk}\geq 0$ are the coupling strengths between the populations $k$ and $j$ and $\theta_p$ are the thresholds. 

We select the parameters so that the space clamped model (when the solutions are independent of $x$) has three equilibria $(u_i,v_i)$ for $i=1,2,3$ with $u_1<u_2<u_3$. When $(u_3,v_3)$ is stable, the network is in an inhibitory stabilized regime \cite{tsodyks_paradoxical_1997,harris_traveling_2018, veltz_periodic_2015}.

\begin{remark}
	It is convenient to rewrite \eqref{eq:nfe} in the more compact way
	\begin{equation}
		\label{eq:nfe-mK}
		\frac{d\mfu}{dt} = - \mathbf L_0\mfu +  \mathbf L_0S\left(\mW\cdot\mfu-\mathbf\theta \right)
	\end{equation}
	where $\mathbf L_0$ is the diagonal operator $\mathbf L_0:=diag(1,\tau^{-1})$, $\theta:=(\theta_e,\theta_i)$ and $\mW_{ij} := a_{ij}K_j$.
\end{remark}

\section{Informal introduction to fronts, pulses and waves}\label{sec:sumup}
We regroup here the different formulations and results concerning the traveling wave solutions that we shall discuss in this paper. For convenience, we rewrite \eqref{eq:nfe} as a Cauchy problem
\begin{equation}\label{eq:cauchy}
	\frac{d\mfu}{dt} = \mathbf F(\mfu;p)
\end{equation}
where $p\in\R^m$ are the parameters. 

\begin{remark}
	It is enlightening to view \eqref{eq:cauchy} as an equivariant Cauchy problem \cite{chossat_methods_2000} with symmetry group composed of the translations of the real line. A traveling wave is thus a relative equilibrium and a modulated traveling wave (see below) is a relative periodic orbit. See \cite{chossat_methods_2000} for more information on relative equilibria / periodic orbits.
\end{remark}

We are interested in a specific type of solutions $(\bar\mv,\bar c)$ called \textit{traveling wave} (TW) with speed $\bar c$ and profile $\bar\mv$, where $\mfu(x,t) = \bar\mv(x -\bar ct)$ and where the limits $\lim\limits_{\xi\to\pm\infty}\bar\mv(\xi)$ are one of the $(u_i,v_i)$. Depending on the limits, we shall call the TW differently. If these two limits are the same $(u_i,v_i)$ for some $i\in\{1,2,3\}$, \textit{i.e.} connecting one of the space clamped stationary states to the same one, we call such solution a \textit{traveling pulse} (TP), see figure~\ref{fig0}~top for an example. A \textit{traveling front} (TF) is a TW where the two limits are different.

\begin{figure}[htbp!]
	\centering
	\includegraphics[width=0.9\textwidth]{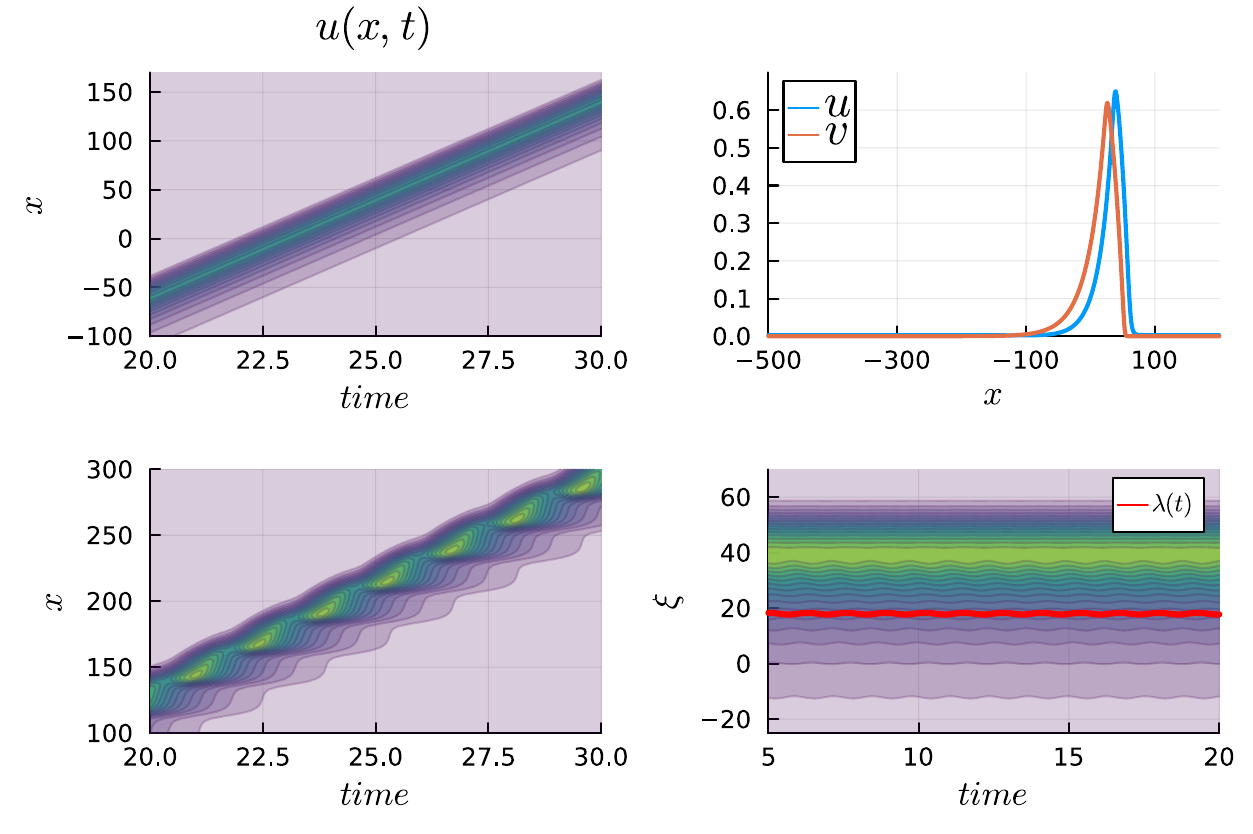}
	\caption{Top left: spatio-temporal plot of a TP. Top right: profile of the TP. Bottom left: spatio-temporal plot of a MTP. Bottom right: MTP in wave frame, solution to \eqref{eq:pddae}. Parameters for top row $\theta_i = 0.3869, \tau = 1.3$. Parameters for bottom row $\theta_i = 0.3869, \tau = 0.82$. The other parameters can be found in section~\ref{section:simus}.}
	\label{fig0}
\end{figure}

\begin{remark}
	The fact that the network is in an ISN regime is crucial for the definition of the different TWs. We would not necessarily find them in a different dynamical regime.
\end{remark}

\begin{remark}
	Note that the numerical methods of section~\ref{sec:nm} work the same for \textit{traveling fronts} (see \cite{faye_traveling_2018}). We will not numerically study them thoroughly in this work.
\end{remark}

It is straightforward to derive an equation satisfied by the TW but it is more convenient to  introduce the so-called \textit{freezing method} \cite{beyn_freezing_2004, rottmann-matthes_stability_2012, beyn_stability_2014} by expressing \eqref{eq:cauchy} in a wave frame $\xi=x-\gamma(t)$ with $\gamma(0)=0$. The wave ansatz $\mfu(x,t) = \mv(x-\gamma(t), t)$ for some phase $\gamma$ and profile $\mv$, yields the equations where we write $\partial:=\partial_\xi$:

\begin{equation*}
	\left\{
	\begin{array}{ccl}
		\frac{d\mv}{dt} &=& \lambda\partial\mv + \mF(\mv;p),\\
		\frac{d\gamma}{dt} &=&\lambda.
	\end{array}\right.
\end{equation*}
These equations are easy to get, they stem from the time derivative of $\mv(x-\gamma(t), t)$.
However, the mapping $\gamma$ is still arbitrary. We fix this degree of freedom by using a phase-condition. For a given reference profile $\hat\mfu$, we look for the profile $\mv$ that minimize $J(\gamma):=\|\hat\mfu(\cdot-\gamma)-\mv \|^2_{\Lp^2}$. This leads to the phase condition: 
\[
0 =  \langle\mv-\hat \mfu, \partial\hat \mfu \rangle_{L^2}.
\]
All in all, we arrive at the differential algebraic equation (DAE):
\begin{equation}\label{eq:pddae}
	\left\{
	\begin{array}{ccl}
		\frac{d\mv}{dt} &=& \lambda\partial\mv + \mF(\mv;p),\\
		\frac{d\gamma}{dt} &=&\lambda,\\
		0 &=&  \langle\mv-\hat \mfu, \partial\hat \mfu \rangle_{L^2}.
	\end{array}\right.
\end{equation}
We observe that the second equation decouples from the other two. 
Given a solution $\mfu$ of \eqref{eq:cauchy}, we can recover the phase $\gamma$ by substituting $\mfu(x + \gamma(t), t)$ in the phase condition and differentiating \textit{w.r.t.} time:
\[
\dot\gamma=g(t,\gamma(t)), \quad\quad g(t,\gamma) := \frac{ \langle\mF(\mfu(t);p), \partial\hat \mfu(\cdot+\gamma) \rangle_{L^2}}{ \langle\mfu(t), \partial\hat \mfu(\cdot+\gamma) \rangle_{L^2}}.
\]
Thus, the only freedom in finding those waves is to fix a reference profile $\hat\mfu$ such that \eqref{eq:pddae} is well posed.

\begin{remark}
	It can be interesting, especially for numerical bifurcation analysis, to view the previous equation \eqref{eq:pddae} as a DAE with mass matrix $\mM = diag(\id_2,0)$
	\begin{equation*}
		\mM\frac{d}{dt} 
		\left[\begin{matrix}
			\mv\\ 
			\lambda
		\end{matrix}\right]
		= 
		\left[\begin{matrix}
			\lambda\partial\mv + \mF(\mv;p)\\
			\langle\mv-\hat \mfu, \partial\hat \mfu\rangle_{L^2}
		\end{matrix}\right].
	\end{equation*}
\end{remark}
\noindent A TW with speed $\bar c$ is thus a stationary point $(\bar\mv,\bar c)$ of \eqref{eq:pddae} with appropriate boundary conditions:
\begin{equation}\label{eq:wave}
	\left\{
	\begin{array}{ccl}
		0 &=& \bar c\partial\bar\mv + \mathbf F(\bar\mv;p)\\
		0 &=&  \langle\bar\mv-\hat \mfu, \partial\hat \mfu \rangle_{L^2}.
	\end{array}\right.
\end{equation}
Finally, we are interested in \textit{modulated traveling waves}\footnote{called MTP for pulses. We require that the limits at $\xi=\pm\infty$ are equal at all time} MTW $(\tilde\mfu,\bar c,T)$ which are $T$-periodic solutions of \eqref{eq:cauchy} satisfying \[\forall t\geq0,\forall x\in\R,\ \tilde\mfu(x,t+T)=\tilde\mfu(x-\bar cT,t)\] or equivalently $\tilde\mv(\xi, t+T)= \tilde\mv(\xi,t)$ in the wave frame $\xi=x-\bar ct$. MTW can thus be seen as $T$-periodic solutions $\tilde\mv$ of 
\begin{equation}\label{eq:mtw_ode}
	\frac{d\mv}{dt} =\bar c\partial\mv + \mathbf F(\mv;p).
\end{equation}
Similarly, a $T$-periodic solution $(\tilde\mw,\tilde\lambda)$ of  \eqref{eq:pddae} also yields a MTW.
Indeed, by definition, we have
\[
\mfu(x,t+T)  = \tilde\mw(x-\tilde\gamma(t+T),t+T)  = \tilde\mw(x-\tilde\gamma(t)-\delta\gamma,t)=\mfu(x+\delta\gamma,t)
\]
where $\delta\gamma\stackrel{see\ \eqref{eq:pddae}}{:=}\int_0^T\tilde\lambda=\bar c T$ defines $\bar c$ as mean value of $\tilde\lambda$. The numerical advantage of the formulation \eqref{eq:pddae} is that the wave speed $\bar c$ is a byproduct but we need to solve a DAE whereas when solving  \eqref{eq:mtw_ode}, the speed $\bar c$ must be known accurately.

\section{Notations and main results}\label{section:mainR}
In the previous section, we described the type of waves that we want to study. Our work is not focused on the very delicate question of existence \cite{ermentrout_existence_1993} of such solutions (some of the most recent methods are nicely presented in \cite{faye_propagation_2019}) except some MTW which can be predicted from a Hopf bifurcation of TW (see theorem~\ref{thm:CM} below). We are interested in the dynamical stability of these solutions. Some of these stability results are very well known \cite{sandstede_stability_2002} albeit mostly in the diffusive case (but see \cite{rottmann-matthes_stability_2012}).

\subsection{Notations}
We start by introducing some notations. 
We denote by $\mathcal L(X)$ the set of continuous linear operators on a banach space $X$.
Let $\ker(\mA)$ be the kernel of a linear operator $\mA$, $\mathcal R(\mA)$ its range and $\mA^*$ its adjoint. 
The spectrum of $\mA$ is denoted by $\Sigma(\mA)$ and the set of eigenvalues is denoted by $\Sigma_p(\mA)$.
$R(z,\mA):=(z\id-\mA)^{-1}$ denotes the resolvent of $\mA$ for $z\in\mathbb C\setminus\Sigma(\mA)$.
Let $\Lp^q :=\mathrm L^q(\R,\mathbb R^2)$ be the space of functions with two components in $\Lp^q(\R,\R)$, $\Hp^m :=\mathrm H^m(\R,\mathbb R^2) := W^{2,m}(\R,\mathbb R^2)$ be the Sobolev spaces \cite{brezis_functional_2010}. We also write $\mathcal M_2(\R)$ the set of real 2 by 2 matrices.
For a function $\mathbf F(\mfu; p)$, we denote by $d\mathbf F(\mfu; p )$ the Frechet differential at $(\mfu, p)$ \textit{w.r.t.} the first variable.

When discussing a solution of \eqref{eq:cauchy}, we write $\mfu(t) := \mfu(\cdot, t)$. Finally, we denote by $(\rho(\gamma)\cdot\mfu)(x):=\mfu(x-\gamma)$ for some $\gamma\in\R$ the (space) translation of $\mfu$.

\subsection{Stability of traveling waves}
\begin{definition}[Orbital stability \cite{thummler_numerical_2005}]\label{def:stabletw}
	A traveling wave $(\bar\mv,\bar c)$ is exponentially and asymptotically orbitally stable if there are $M\geq0, \alpha>0$ such that for all $\epsilon>0$ sufficiently small, there is $\delta>0$ such that for each solution $\mfu$ of \eqref{eq:cauchy} with $\norm{\mfu(0)-\bar\mv(\cdot+\gamma_0)}\leq \delta$ for some $\gamma_0\in\mathbb R$, there exists a phase shift $\gamma\in\R$ such that
	\[
	\forall t\geq0,\ \norm{\mfu(\cdot,t) - \bar\mv(\cdot-\bar ct+\gamma)}\leq \epsilon.
	\]
	Additionally,
	\[
	\norm{\mfu(\cdot,t) - \bar\mv(\cdot-\bar ct+\gamma)}\leq Me^{-\alpha t}.
	\]
\end{definition}
The orbital stability of $(\bar\mv,\bar c)$ is determined (see \cite{sandstede_stability_2002} for reaction-diffusion equations) by the spectrum of 
\begin{equation}\label{eq:oplinstab}
	\mA:=\bar c\partial + d \mathbf F(\bar\mv;p).
\end{equation}
The zero eigenvalue is always present because $\partial\bar\mv$ in the kernel of $\mA$ as a consequence of the translation symmetry. If $0$ is a simple eigenvalue and the rest of spectrum lies in $\{z\in\mathbb C\ : \ \Re(z)\leq w\}$ for some $w<0$, then the TW is (orbitally) stable. This is an example of the \textit{principle of linearized stability}. We shall see that this holds for the neural field equations \eqref{eq:nfe}. Note that this stability result was conjectured in \cite{faye_traveling_2018} in a slightly different setting.

We now collect the different hypotheses concerning the components of \eqref{eq:nfe-mK}.

\begin{hypothesis}\label{hyp:S}
	$S\in C^{r+1}(\R,[0,1])$ with $r\geq 2$ is increasing with all derivatives bounded. 
\end{hypothesis}
\begin{hypothesis}\label{hyp:barv}
	We assume that
	\begin{enumerate}
		\item $\bar c >0$ (\textit{w.l.o.g.} using the reflection symmetry),
		\item $\partial\bar\mv\in\Lp^2$,
		\item $\hat\mfu \in H^2$,
		\item $ \langle\partial\hat\mfu,\partial\bar\mv \rangle_2\neq 0$.
	\end{enumerate}
\end{hypothesis}
Additionally, because of the reflection symmetry $\kappa\cdot\mfu(x) := \mfu(-x)$ of equation \eqref{eq:nfe} for even spatial connectivity kernels $K_p$, if $(\bar\mv,\bar c)$ is a solution, then so is $(\kappa\cdot\bar\mv,-\bar c)$ which makes Hypothesis~\ref{hyp:barv}.1 less restrictive.

\begin{remark}
	Choosing $\hat\mfu$ in $\ker(\mA^*)$ so that $\langle \hat\mfu, \partial\bar\mv\rangle=1$ would simplify the analysis but in practical applications,  $\bar\mv$ is not known beforehand. 
\end{remark}

\begin{hypothesis}\label{hyp:W}
	$K_k\in \mathrm W^{1,1}(\R, \R^+)$ for $k\in\{e,i\}$.
\end{hypothesis}

\begin{theorem}\label{th:linstab}
	Let us consider a traveling wave $(\bar\mv,\bar c)$.
	Grant \ref{hyp:S}, \ref{hyp:barv} and \ref{hyp:W}. 
	Assume that $\ker\mA = span\{\partial\bar\mv\}$. 
	Further assume that the non-zero eigenvalues  of $\mA$ are in the set $\{\Re<w\}\subset\mathbb C$ for some $w<0$.
	Then, the traveling wave is  exponentially and asymptotically orbitally stable in $\Lp^2$.
\end{theorem}
This results implies that the norm in definition~\ref{def:stabletw} is the $\Lp^2$ one. In particular, for traveling pulses (which do not belong to $\Lp^2$), we only allow for neighborhoods in $\Lp^2$.

The theorem is especially useful in that we only have to characterize the eigenvalues of $\mA$ close to the imaginary axis and not the full spectrum.

\begin{remark}
	We prove stability of the TW for the DAE formulation \eqref{eq:pddae} in theorem~\ref{thm:pstb-dae} which implies theorem~\ref{th:linstab}.
\end{remark}

\subsection{Bifurcations of traveling waves}

The goal of this section is to provide a bifurcation theory for TW. To this end, we prove the existence of an (invariant) center manifold of finite dimension for \eqref{eq:pddae}. Then, the study of the dynamics on this invariant manifold through the use of normal form theory \cite{iooss_bifurcation_1979} gives access to a bifurcation theory for TWs.

We introduce the projector $\mP$ on $\Lp^2$ onto $\partial\hat\mfu^\perp$ along $\partial\bar\mv$
\begin{equation}
	\mP\cdot\sfv := \sfv-\frac{  \langle\sfv, \partial\hat \mfu \rangle_{L^2}}{\langle\partial \bar\mv, \partial\hat \mfu \rangle_{L^2}}\partial\bar\mv
\end{equation}
which is well defined under Hypothesis~\ref{hyp:barv}.

\begin{theorem}\label{thm:CM}
	Let us consider a traveling wave $(\bar\mv,\bar c)$.
	Grant \ref{hyp:S}, \ref{hyp:barv} and \ref{hyp:W}. 
	Assume that $\sigma_c:=\Sigma(\mP\mA)\cap i\R\neq\emptyset$ and denote by
	$\mathbb E_c$ the sum of the generalized eigenspaces of $\mP\mA$ for the zero real part eigenvalues $\sigma_c$.
	Define $\mathcal Z:=\mP\Hp^1$.
	Then, there exists a map $\Psi\in C^r(\mathbb E_c\times\R^m, \mathcal Z)$ with $r\geq 2$ such that
	\[
	\Psi(0;0)=0,\ d\Psi(0;0)=0
	\]
	and a neighborhood $\mathcal O_\mfu\times\mathcal O_p$ of $(0,0)$ in $\mathcal Z\times \R^m$ such that for $p\in\mathcal O_p$, the manifold
	\[
	\mathcal{M}(p)=\left\{\mv_0+\Psi\left(\mv_0; p\right) ; \mv_0 \in \mathbb E_c\right\}\subset\mathcal Z
	\]
	has the following properties:
	\begin{itemize}
		\item $\mathcal{M}(p)$ is locally invariant, i.e., if $\mv$ is a solution to \eqref{eq:ioossform} satisfying $\mv(0)\in\mathcal M(p)\cap\mathcal O_\mfu$ and $\mv(t)\in\mathcal O_\mfu$ for all $t\in[0,T]$, then $\mv(t)\in\mathcal M(p)$ for all $t\in[0,T]$.
		\item $\mathcal{M}(p)$ contains the set of bounded solutions of \eqref{eq:ioossform} staying in $\mathcal O_\mfu$ for all $t\in\R$, \textit{i.e.}, if $\mv$ is a solution to \eqref{eq:ioossform} satisfying $\mv(t)\in\mathcal O_\mfu$ for all $t\in\R$, then $\mv(0)\in\mathcal{M}(p)$.
	\end{itemize}
\end{theorem}

This theorem is the basis for the study of the Hopf bifurcation points that appear in the simulations in section~\ref{section:simus}, it was also mentioned in \cite{sandstede_essential_2001, sandstede_structure_2001} and in \cite{haragus_local_2011}. Indeed, assume that $\pm\i\omega\in\Sigma(\mA)$ for some $\omega>0$ with associated eigenvectors $\zeta,\bar\zeta$ and no other eigenvalue is present on the imaginary axis. Further assume that $\ker\mA = span\{\partial\bar\mv\}$. Then, the previous theorem yields a two-dimensional invariant manifold on which we can apply normal form theory. Generically, we expect to see the following dynamics
\[
\dot z = (\alpha +\i\omega)z+\beta z|z|^2+h.o.t,\quad \alpha,\beta\in\mathbb C,
\]
where $\mv(t) = z(t)\zeta + \bar z(t)\bar\zeta + h.o.t.$ and $h.o.t.$ denotes higher order terms.
The analysis of bifurcations like the Fold, Hopf or Bogdanov-Takens ones is based on the previous theorem and is left for future analysis. Note that these bifurcations of TW appear in the numerical analysis (section~\ref{section:simus}) of the neural fields model \eqref{eq:nfe}.

\subsection{Linear stability of modulated traveling waves}

\begin{definition}[Orbital stability of MTW]\label{def:stablemtp}
	A modulated traveling wave $(\tilde\mv,\bar c, T)$ is exponentially and asymptotically orbitally stable if there are $M\geq0, \alpha>0$ such that for all $\epsilon>0$ sufficiently small, there is $\delta>0$ such that for each solution $\mfu$ of \eqref{eq:cauchy} with $\norm{\mfu(0)-\tilde\mv(\cdot-\bar ct+\gamma_0,\phi_0)}\leq \delta$ for some $\gamma_0,\phi_0\in\R$, there exist shifts $\gamma,\phi\in\R$ such that
	\[
	\forall t\geq0,\ \norm{\mfu(\cdot,t) - \tilde\mv(\cdot-\bar ct+\gamma,t+\phi)}\leq \epsilon.
	\]
	Additionally, 
	\[
	\norm{\mfu(\cdot,t) - \tilde\mv(\cdot-\bar ct+\gamma,t+\phi)}\leq Me^{-\alpha t}.
	\]
\end{definition}

The stability of MTW for reaction-diffusion equations (see \cite{sandstede_stability_2002,sandstede_structure_2001}) is determined by the (Floquet multipliers) spectrum of the linearized time-$T$ map associated with \eqref{eq:mtw_ode}. If $1$ is a Floquet multiplier (one for time / space equivariance) with geometric and algebraic multiplicity two and the rest of spectrum lies in a compact subset of the open unit disk, then the MTW is stable. This is yet another example of the \textit{principle of linearized stability} for waves in reaction-diffusion equations.

Given a MTW $\tilde\mv$ of period $T$ solution to \eqref{eq:mtw_ode}, we consider the monodromy operator $\mathbf S$ which is the evolution operator of the linearized equation around $\tilde\mv$ at time $T$. More precisely, if we denote by $\mathbf U(t,s)\mv_s$ the solution of the linearized problem
\begin{equation}\label{eq:linper}
	\frac{d\mv}{dt} = \bar c \partial\mv + d\mF(\tilde\mv(t);p)\mv :=\mA(t)\cdot\mv,
\end{equation}
with initial condition $\mv(s) = \mv_s\in\Lp^2$, then we have $\mathbf S:=\mathbf U(T,0)$. The spectral properties of $\mathbf S$ control the long time behavior of the solutions of \eqref{eq:linper} thanks to the identity
\[
\mathbf U(nT+s,0) = \mathbf S^n\mathbf U(s,0),\quad\forall n\in\mathbb N, s\geq 0.
\]

\begin{proposition}\label{prop:stab-mtw}
	Grant \ref{hyp:W}.
	Let us consider a modulated traveling wave $(\tilde\mv,\bar c, T)$ of period $T$ solution to \eqref{eq:mtw_ode}.
	Assume that $\tilde\mv\in H^1_{per}([0,T], \Hp^1)$ \textit{i.e.} it is $T$-periodic. Then we have the following properties.
	\begin{enumerate}
		\item The intersection of the spectrum $\Sigma(\mathbf S)$ with the unit circle consists of a finite (or empty) set of isolated eigenvalues with finite algebraic multiplicity.
		\item If $\tilde\mv\in H^2_{per}([0,T], \Hp^1)$, then $1\in\Sigma_p(\mathbf S)$ is at least a geometrically double eigenvalue provided that $\tilde\mv(\xi,t)$ is not equal to $h(\xi- \alpha t)$ for $\alpha\in\R$ and $h$ periodic. In that case, $\partial_t\tilde\mv(\cdot, 0), \partial\tilde \mv(\cdot, 0)\in \ker(\mathbf S-I)$.
		\item Assume that $1\in\Sigma_p(\mathbf S)$ is an eigenvalue with geometric and algebraic multiplicities two and the other eigenvalues are in the open unit disk. Then, for all $s\in\mathbb R$ and for all $\mv\in span(\partial_t\tilde\mv(\cdot, s), \partial\tilde \mv(\cdot, s))^\perp$
		\begin{equation}
			\|\mU(t,s)\mv\|_{2} \to 0 \text{ as }t\to\infty.
		\end{equation}
	\end{enumerate}
\end{proposition}

Note that the first part implies a spectral gap.
The crux of the previous proposition is that the linear stability of MTW is controlled by the \textit{eigenvalues} of the monodromy operator $\mathbf S$. We conjecture that this should remain true for the nonlinear stability although the proof would take us too far.

\begin{conjecture}[Nonlinear stability] Consider a MTW $(\tilde\mv,\bar c,T)$ solution to \eqref{eq:mtw_ode} where $\tilde\mv\in H^2_{per}([0,T], \Hp^1)$ and $\bar c>0$. Assume that apart from the eigenvalue $1$, the {eigenvalues} of $\mathbf S$ are included in a compact subset of the open unit disk. Assume that $1$ is an eigenvalue with geometric and algebraic multiplicities two. Then, the MTW is exponentially and asymptotically orbitally stable.
\end{conjecture}

\section{Stability of traveling waves}\label{section:nlstat}

We prove some of the results of the previous section. Our general strategy is adapted from \cite{rottmann-matthes_stability_2012} but for a different setting. For a different method, see \cite{faye_traveling_2018} albeit for TF in lattice neural fields.
Let us consider a TW $(\bar\mv, \bar c)$.
In our notations \eqref{eq:nfe-mK} and \eqref{eq:oplinstab}, we find that
\begin{equation}\label{eq:mA-NF}
\mA = \bar c\partial -\mathbf L_0+\mathbf L_0S'\left(\mW\cdot \bar\mv+\theta\right)\mW:= \bar c\partial -\mathbf L_0+\tilde\mW.
\end{equation}
Let $\mB := -\mathbf L_0+\tilde\mW$ so that we can write \eqref{eq:oplinstab}
\[
\mA = \bar c\partial + \mB.
\]

\subsection{Linear stability of traveling waves}
We start with a sequence of results to prove the linear stability of the TW in $\Lp^2$.
\begin{lemma}\label{lem:T}
	Let $(\bar\mv ,\bar c)$ be a TW solution to \eqref{eq:nfe-mK}.
	Grant \ref{hyp:S}, \ref{hyp:barv} and \ref{hyp:W}. 
	Then $\mB\in\mathcal L(\Lp^2)$ and $(\mA, D(\mA))$ with domain $D(\mA) := \Hp^1\subset\Lp^2$ generates a strongly continuous semigroup $(\mT(t))_{t\geq 0}$ on $\Lp^2$.	
\end{lemma}
\begin{proof}
	\rem{Brezis 4.15 Young inequality $L^1\star L^p \subset L^p$ gives $\mB\in\mathcal L(\Lp^2)$. }
	We first observe that $K_k\in\mathcal L(L^2)$ by hypothesis~\ref{hyp:W}.
	 This is a consequence of the Young's inequality for convolutions \cite{brezis_functional_2010} as $K_k\in\Lp^1$.
	 It then follows that $\mW \in\mathcal L(L^2)$ by hypothesis~\ref{hyp:W} and so are $\mB$ and $\tilde \mW$ by hypothesis~\ref{hyp:S}.
	 
	Next,  it is well known \cite{engel_one-parameter_2000} that $\bar c\partial$ with domain $D(\mA)$ and with $\bar c\neq 0$ (by hypothesis~\ref{hyp:barv}) generates a strongly continuous semigroup on $\Lp^2$, called the left translation semigroup. Indeed, $\dot \mfu = \bar c\partial \mfu$ has general solution $\mfu_0(x +\bar c t)$ for some $\mfu_0\in \Lp^2$. Then $\mA =\bar c\partial+\mathbf B$ is a bounded perturbation of $\bar c\partial$ because $\mathbf B\in\mathcal L(\Lp^2)$. It therefore \cite{engel_one-parameter_2000}[Theorem II.1.3] generates a strongly continuous semigroup on $\Lp^2$ as well.
\end{proof}
Next, we look at the spectrum of $\mA$ defined in \eqref{eq:mA-NF}.
Its computation poses several challenges. Indeed, the classical method \cite{henry_geometric_1981, faye_traveling_2018} is to view \eqref{eq:mA-NF} as a Fredholm operator for which the structure of the spectrum is well known. However this is not completely straightforward and we provide another argument. We rely on the notion of \textit{quasi-compact} operators \cite{engel_one-parameter_2000} for which the spectrum is well understood.

\begin{definition}[Quasi-compact semigroup \cite{engel_one-parameter_2000}]
	A strongly continuous semigroup $(\mT(t))_{t\geq 0}$ on a Banach space $X$ is called \textit{quasi-compact} if
	\[
	\lim_{t\to\infty}\inf\{\|\mT(t) -\mathbf K\|\ :\ \mathbf K\in\mathcal L(X),\ \mathbf K\text{ compact}    \} = 0.
	\]
\end{definition}

We introduce cutoff functions $\chi_n\in C^\infty(\R,[0,1])$ for $n\geq 0$ that are equal to one on $[-n,n]$ and zero outside $[-n-1,n+1]$.

\begin{definition}[Growth bound \cite{engel_one-parameter_2000}]
	For a strongly continuous semigroup $(\mT(t))_{t\geq 0}$ on $\mathcal X$, the growth bound is defined as
	\[
	\omega_0(\mT) = \inf\{\omega\in\R\ |\ \exists M_\omega\geq 1\text{ such that }\forall t\geq 0,\  \norm{\mT(t)}\leq M_\omega e^{\omega t} \}.
	\]
\end{definition}

\begin{lemma}\label{lem:T0}
	Let $(\bar\mv ,\bar c)$ be a TW solution to \eqref{eq:nfe-mK}.
	Grant \ref{hyp:barv}. 
	The operator $\mA_0 = \bar c \partial -\mathbf L_0$ with domain $D(\mA_0)=D(\mA)$ generates a strongly continuous semigroup $(\mT_0(t))_{t\geq 0}$ on $\Lp^2$ such that 
\begin{equation}\label{eq:bound-T0}
\norm{\mT_0(t)}_2 \leq \exp\left( -\min(1,\tau^{-1})t\right).
\end{equation}
Its growth bound $\omega_0(\mT_0)$ is such that 
	\[
	\omega_0(\mT_0)\leq -\min(1,\tau^{-1})<0.
	\]
\end{lemma}	
\begin{proof}
	It is straightforward to check that 
	$\mT_0(t)\mfu = e^{-\mathbf L_0t}\mfu(x +\bar c t)$.
	The lemma follows easily.
\end{proof}

\begin{lemma}
	Let $(\bar\mv ,\bar c)$ be a TW solution to \eqref{eq:nfe-mK}.
	Grant \ref{hyp:S}, \ref{hyp:barv} and \ref{hyp:W}. 
	There exists $n\in\mathbb N$ large enough such that the semigroup $(\mT_n(t))_{t\geq 0}$  generated by 
	\[
	\bar c\partial -\mathbf L_0+(1-\chi_n)\tilde\mW
	\]
	has growth bound 
	\begin{equation}\label{eq:omega-Tn}
	\omega_0(\mT_n) < 0. 
	\end{equation}
	As a consequence, it is a quasi-compact semigroup.
\end{lemma}
\begin{proof}
	The operator $\mathbf B_n:=(1-\chi_n)\tilde\mW$ is bounded on $\Lp^2$ with operator norm converging to zero as $n\to\infty$. From \cite{engel_one-parameter_2000}[Theorem III.1.3], $\mA_0+\mathbf B_n$ generates a semigroup with growth bound smaller than $\omega_0(\mT_0) + \norm{\mB_n}$. The growth bound estimate \eqref{eq:omega-Tn} follows easily from lemma~\ref{lem:T0}. By \cite{engel_one-parameter_2000}[proposition IV.2.10] and \cite{engel_one-parameter_2000}[proposition V.3.5], $(\mT_n(t))_{t\geq 0}$ is quasi-compact.
\end{proof}

\begin{proposition}\label{prop:linstabTW}
Let $(\bar\mv ,\bar c)$ be a TW solution to \eqref{eq:nfe-mK}.
Grant \ref{hyp:S}, \ref{hyp:barv} and \ref{hyp:W}. 
Then, the semigroup $(\mT(t))_{t\geq 0}$ is quasi-compact. As consequence, the set $\{\lambda\in \Sigma(\mA)\ |\ \Re\lambda \geq 0\}$ is finite (or empty) and consists of isolated poles of $R(\cdot, \mA)$ with finite algebraic multiplicity. The finite dimensional semigroup $\tilde\mT$ on these eigenspaces corresponding to these poles satisfies
\[
\norm{\mT(t) - \tilde\mT(t)}\leq M e^{-\epsilon t}
\]
for some $M\geq 1,\epsilon>0$.
\end{proposition}
\begin{proof}
	\rem{$\chi_n\mW$ est compact puis $\mathbf L_0S'(\cdot)\cdot\chi_n\mW=\mK_n$ est compact.}
	The operator $\mK_n := \chi_n\tilde\mW$ is bounded on $\Lp^2$. It is compact by \cite{brezis_functional_2010}[corollary 4.28]. The operator	$\mA = \mA_0+\mB_n+\mK_n$ thus generates a quasi-compact semigroup as consequence of \cite{engel_one-parameter_2000}[proposition V.3.6]. The rest of the lemma is \cite{engel_one-parameter_2000}[theorem V.3.7].
\end{proof}

\subsection{Reformulation as a DAE}
Next, we re-write the freezed problem \eqref{eq:pddae} near a traveling wave $(\bar\mv,\bar c)$. Using the ansatz, $\mv=\bar\mv+\sfv$, $\lambda = \bar c+\mu$, we express \eqref{eq:pddae} as
\begin{equation}\label{eq:pddae-linstab}
	\left\{
	\begin{array}{ccl}
		\frac{d\sfv}{dt} &=& \mA\sfv+\mu\partial\bar\mv+ \mG(\sfv,\mu;p)\\
		0 &=&  \langle\sfv, \partial\hat \mfu \rangle_{L^2}
	\end{array}\right.
\end{equation}
where $\mG(\sfv,\mu;p) := \mu\partial\sfv+\tilde \mF(\sfv;p)-\mB\cdot\sfv$ collects the quadratic terms and 
\begin{equation}\label{eq:tildeF}
\tilde \mF(\sfv;p) := \mF(\bar\mv+\sfv;p)-\mF(\bar\mv;p).
\end{equation}
We find that $\tilde\mF$ is regular on $\Lp^2$. 
\begin{lemma}\label{lemma:FCk}
	Grant \ref{hyp:S} and \ref{hyp:W}, then $\tilde\mF(\cdot;p)\in  C^r(\Lp^2,\Lp^2)$.
\end{lemma}
\begin{proof}
	See \ref{proof:lemma47}.
\end{proof}
We now adapt the strategy of \cite{thummler_numerical_2005} to transform the DAE \eqref{eq:pddae-linstab} into a Cauchy problem, \textit{i.e.} without the constraint.
Note that $\langle\mP\cdot\sfv,\partial\hat\mfu\rangle = 0$ for all $\sfv\in\Lp^2$ and $\mP\cdot\partial\bar\mv= 0$. 
 Differentiating the constraint in \eqref{eq:pddae-linstab} yields
\begin{equation}\label{eq:projected}
	\frac{d\sfv}{dt} = \mP\left(\mA\sfv+\mG(\sfv,\mu(\sfv);p) \right)
\end{equation}
with 
\begin{equation}
	\mu(\sfv) := -
	\frac{\langle\partial\hat\mfu, \bar c\partial\sfv+\tilde\mF(\sfv;p) \rangle}{\langle \partial\hat\mfu,\partial\bar\mv + \partial\sfv\rangle}.
\end{equation}
Using an integration by parts under Hypothesis~\ref{hyp:barv}, one can show that $\mu\in  C^r(\Lp^2,\R)$ using lemma~\ref{lemma:FCk}. However, we still face a difficulty: the nonlinear term $\mu(\sfv)\partial\sfv$ is defined on the domain $D(\mA)$ of $\mA$. This makes \eqref{eq:projected} a quasilinear problem which is not easy to handle \cite{pazy_semigroups_1983}. We thus look for a different formulation and rewrite \eqref{eq:projected} as
\begin{equation}\label{eq:cpmu}
	\frac{d\sfv}{dt} = \mP\left((\bar c+\mu(\sfv))\partial\sfv+\tilde\mF(\sfv;p) \right).
\end{equation}
Because we are interested in the local behavior near $\mv = 0$, we can use a time change \cite{grabosch_cauchy_1990,drogoul_exponential_2021} to get a semilinear formulation which is much easier to analyze:
\begin{equation}\label{eq:cpmu-tc}
	\frac{d\sfv}{dt} = \mP\left(\partial\sfv+\frac{1}{\bar c+\mu(\sfv)}\tilde\mF(\sfv;p) \right).
\end{equation}

\rem{Dans \cite{grabosch_cauchy_1990}, il faut prouver que $A_0+dB(0)$ est stable. on trouve $dB(0) = -\frac{d\gamma(0)}{\gamma(0)^2}F(0)+\frac{1}{\gamma(0)}dF(0)$}

\subsection{Principle of linearized stability for the DAE}
\begin{theorem}\label{thm:pstb-dae}
	Let us consider a traveling wave $(\bar\mv,\bar c)$.
	Grant \ref{hyp:S}, \ref{hyp:barv} and \ref{hyp:W}. 
	Assume that $\ker\mA = span\{\partial\bar\mv\}$. 
	Further assume that the non-zero eigenvalues are in the set $\{\Re<w\}$ for some $w<0$.
	Then $0$ in $\mathcal R(\mP)$ is locally exponentially stable for \eqref{eq:cpmu}.
\end{theorem}
\begin{proof}
	$\tilde F$ and $\mu$ are Frechet differentiable on $\Lp^2$ by lemma~\ref{lemma:FCk}. The theorem is then a direct consequence of \cite{grabosch_cauchy_1990}[Theorem 5.1] which states a principle of linearized stability for the quasilinear formulation \eqref{eq:cpmu}, this is proved using the time change conversion to the semilinear formulation \eqref{eq:cpmu-tc}. In practice, one has to introduce a cutoff function $\rho:\mathbb R\to\mathbb R^+_*$ which is the identity map near $\bar c$ to ensure that the time change does not change the dynamics, see \cite{drogoul_exponential_2021} for more details.
	The only remaining point is to show that $0$ is exponentially stable for the linear equation $\frac{d}{dt}\sfv=\mP\mA\sfv$. 
	
	Let us denote by $\mP_0$ the spectral projector on $\ker(\mA)=span(\partial\bar\mv)$. 
	We recall that $(\mT(t))_{t\geq 0}$ is the semigroup generated by $\mA$ (see lemma~\ref{lem:T}).
	By \cite{engel_one-parameter_2000}[theorem~V.3.1], one has
	\[
	\mT(t) =  \mP_0 + \mR(t)
	\]
	where $\norm{\mR(t)}\leq Me^{wt}$ for some $M>0$ and $0>w>\sup\Sigma(\mA)\setminus\{0\}$. One then gets
	$\mP\mT(t)\mP = \mP \mP_0\mP+\mP\mR(t)\mP$. From $ \mP \mP_0=0$, we obtain $\norm{\mP\mT(t)\mP}\leq Me^{wt}$ which concludes the proof.
\end{proof}

\subsection{Proof of theorem~\ref{th:linstab}}
Using the above theorem, we can deduce the exponential orbital stability of the wave $\bar\mv$ in $\Lp^2$. Indeed, it implies that $\int_0^t\mu(\sfv(s))ds$ converges as $t\to\infty$.

\section{Bifurcation theory of traveling waves}\label{section:cm}

Using the analysis of section~\ref{section:nlstat}, we obtain the equivalent equation \eqref{eq:projected} of the DAE \eqref{eq:pddae-linstab} for $\mv,\lambda$ small. We introduce the spaces $\mathcal X:=\mP\Lp^2=\mathcal Y$, $\mathcal Z=\mP\Hp^1$, the linear operator $\mA_\mP:=\mP\mA$ with domain $\mathcal Z$ and the nonlinearity
\[
\mR(\sfv;p):=\mP\mG(\sfv,\mu(\sfv);p).
\]
We rewrite \eqref{eq:projected} as 
\begin{equation}\label{eq:ioossform}
	\frac{d\sfv}{dt}=\mA_\mP\sfv+\mR(\sfv;p).
\end{equation}
We note that the effect of the projector $\mP$ is to remove the eigenvector $\partial\bar\mv$ from the kernel of $\mA$.
\begin{lemma}\label{lemma:sigmaAp}
	The spectrum of $\mA_\mP$ is the same of that of $\mA$ except at zero: $\rho(\mA)\setminus\{0\}=\rho(\mA_\mP)\setminus\{0\}$. The effect of the projector $\mP$ is to remove the eigenvector $\partial\bar\mv$ from the kernel $\ker(\mA)$.
\end{lemma}
\begin{proof}
	Recall that to characterize the spectrum, we have to check existence and continuity of the resolvent.
	For convenience, let $\psi:=\partial\hat\mfu$ and $\phi:=\partial\bar\mv$ so that $\mP\cdot \sfv = \sfv-\frac{  \langle\sfv,\psi \rangle_{L^2}}{\langle\phi, \psi \rangle_{L^2}}\phi$. We compute the resolvent set of $\mA_\mP$. 
	
	For $z\neq 0$, solving $(zI-\mA_\mP)\mv=\mP\cdot\mr$ with $\mu = -\langle \psi,\phi\rangle^{-1}\langle \psi,\mA\mv\rangle$ is equivalent to solving in $(\mv,\mu)$ the bordered problem
	\begin{equation}\label{eq:borderedLS}
		\left\{
		\begin{aligned}
			(z\id-\mA)\mv-\mu\phi&=\mP\cdot\mr \\
			\langle\psi, \mv\rangle & =0.
		\end{aligned}\right.
	\end{equation}
	Following \cite{thummler_numerical_2005}[lemma 1.21], for $z\in\rho(\mA)\setminus\{0\}$, we can solve  \eqref{eq:borderedLS} as follows
	\[
	\mv = R(z,\mA)\left(\mP\cdot\mr+\mu\phi\right).
	\]
	The constraint $\langle \psi,\mv\rangle=0$ gives $\mu=-\langle \psi,R(z,\mA)\phi\rangle^{-1}\langle \psi,R(z,\mA)\mP\cdot\mr\rangle$ so that $\mv=\mQ_z R(z,\mA)\mP\cdot\mr$ with 
	\[
	\mQ_z\mv := \mv - \frac{\langle\psi,\mv\rangle}{\langle \psi,R(z,\mA)\phi\rangle}R(z,\mA)\phi.
	\]
	To simplify the expression, we expand $R(z,\mA)\left(z\id-\mA\right)\phi=\phi$ which gives, as $\phi\in\ker(\mA)$, $z\langle \psi,R(z,\mA)\phi\rangle =\langle\psi,\phi\rangle\neq 0$ by hypothesis~\ref{hyp:barv}. Hence, if $z\neq 0$, we find
	\begin{equation}\label{eq:Q}
		\mQ_z\mv := \mv - z\frac{\langle\psi,\mv\rangle}{\langle \psi,\phi\rangle}R(z,\mA)\phi.
	\end{equation}
	The continuity of $\mQ_z R(z,\mA)\mP$ implies that $z\in\rho(\mA_\mP)\setminus\{0\}$. 
	It is straightforward to check the other inclusion $\rho(\mA_\mP)\setminus\{0\}\subset \rho(\mA)\setminus\{0\}$.
	This concludes the proof of the lemma.
\end{proof}

\subsection{Proof of theorem~\ref{thm:CM}}
\begin{proof}	
	This is a direct consequence of \cite{haragus_local_2011}[theorem II.3.3]. 
	We first observe that $\mA_\mP\in\mathcal L(\mathcal Z, \mathcal X)$. By lemma~\ref{lemma:sigmaAp} and proposition~\ref{prop:linstabTW}, the spectrum $\Sigma(\mA_\mP)$ is discrete near $\i\R$, there is $\gamma>0$ such that $\Sigma(\mA)\setminus\i\R$ is included in $\{z\ :\ \Re(z)<-\gamma\}$.
	
	Next, we note that $\mR\in C^r(\Hp^1,\Lp^2)$ for $r\geq 2$ by lemma~\ref{lemma:FCk}; this implies that $\mR\in C^r(\mathcal Z,\mathcal Y)$. One also has $\mR(0;0)=0$, $d\mR(0;0)=0$. The only point left to prove are resolvent estimates, see \cite{haragus_local_2011}[theorem~2.20] as provided by the next lemma.
	
	\begin{lemma} 
	There are $\omega_0>0, M>0$ such that for all $\omega\in\R$ with $|\omega|>\omega_0$, we have
	\[
	\norm{\left(\i\omega \id-\mA_\mP\right)^{-1}}_{\mathcal L(\mathcal X)}\leq \frac{M}{|\omega|}.
	\]
	\end{lemma}
	\begin{proof}
	Let us prove the estimate for $\mA$ on $\Lp^2$, the case of $\mA_\mP$ is very similar at the price of more complicated notations. It generates the strongly continuous group 
	$\mT_\mA(t)\mv = e^{\mB t}\mv(x +\bar c t)$ with bound $\norm{\mT_\mA(t)}_2\leq e^{\norm{\mB}\cdot |t|}$. From \cite{engel_one-parameter_2000}[Theorem~II.3.11-c], we have for all $|\omega|>\norm{\mB}$:
	\[
	\norm{\left(\i\omega \id-\mA\right)^{-1}}_{\mathcal L(\Lp^2)}\leq \frac{1}{|\omega|-\norm{B}}.
	\]
	 The lemma follows using $M=\frac12$ and $\omega_0 = 2\norm{B}$.
	\end{proof}
	This concludes the proof of the theorem.
\end{proof}

\section{Spectral analysis of modulated traveling waves}\label{section:mtp-nlstab}
Let $\tilde\mv$ be a MTW defined as a periodic solution to the Cauchy problem \eqref{eq:mtw_ode}, it is natural to introduce the linearized problem \eqref{eq:linper} around $\tilde\mv$.
\subsection{Floquet multipliers}
We describe the Floquet multipliers in infinite dimensions.
Given a MTW $\tilde\mv$ of period $T$, we study \eqref{eq:linper}. We have
\begin{equation}
	\label{eq:nonauto}
	\mA(t) = \bar c\partial -\mathbf L_0 + \mathbf L_0S'\left(\mW\cdot \tilde\mv(t)+\theta\right)\mW:= \bar c\partial -\mathbf L_0+\tilde\mW(t).
\end{equation}
We note that $\mA(\cdot)$ is $T$ periodic and $D(\mA(t)) = \Hp^1$ is independent of $t$. We first show that the non-autonomous linear problem \eqref{eq:linper}.

\begin{lemma}\label{lemma:evolfamily}
	Assume that $\tilde\mv\in  C^0(\R,\Lp^2)$ and grant \ref{hyp:W}. 
	Then, there is a unique exponentially bounded evolution family $(\mathbf U(t, s))_{ t\geq s}$ on $\Lp^2$ solution to \eqref{eq:linper}. More precisely, \eqref{eq:linper} with initial condition $\mv(s) = \mv_s\in\Lp^2$, has a unique solution given by $\mv(t) = \mathbf U(t,s)\mv_s$ and there are $M_{\mathbf U}\geq1, w\in\R$ such that for all $t\geq s$, 
	\begin{equation}\label{eq:bound-U}
	\norm{\mathbf U(t,s)}_{\mathcal L(\Lp^2)}\leq M_{\mathbf U}e^{w(t-s)}.	
	\end{equation}
\end{lemma}
\begin{proof}
	The operator $\bar c\partial -\mathbf L_0$ with domain $\mathrm H^1$ generates an exponentially bounded evolution family on $\Lp^2$, it is actually $(\mathbf T_0(t-s))_{t\geq s}$ (see lemma~\ref{lem:T0}). The lemma follows from $\tilde\mW(\cdot)\in\mathcal C_b(\R,\mathcal L(\Lp^2))$ and \cite{engel_one-parameter_2000}[corollary VI.9.20].
\end{proof}

By uniqueness, we have that $\mathbf U(t+T,s+T) = \mathbf U(t,s)$. To this evolution family, we associate the \textit{monodromy operator} 
\[
\mathbf S := \mathbf U(T,0)\in\mathcal L(\Lp^2),
\]
which is a restriction of the \textit{period map} $\mathbf S(t):=\mathbf U(t+T,t)$. One can show that $\Sigma(\mathbf S(t))\setminus\{0\}$ is independent of $t$, see \cite{henry_geometric_1981}. The non-zero eigenvalues of $\mathbf S$ are called the \textit{Floquet multipliers} \cite{iooss_bifurcation_1979}.

\begin{remark}
One is often interested in the \textit{Floquet exponents} $\mathcal E$ which is the spectrum of 
$-\frac{d}{dt} + \mA(t)$  in $\Hp^1_{per}([0,T],\Lp^2)$. Using the evolution family $e^{\sigma(t-s)}\mU(t,s)$ associated with $\mA(t)+\sigma\id$, one has
\rem{On peut faire mieux, Koch-medina, Chicone}
\[
\{\exp\left(zT\right),\ z\in\Sigma_p\left(-\frac{d}{dt} + \mA(t)\right) \}=\Sigma_p(\mathbf S)\setminus\{0\}
\]
and so Floquet exponents and multipliers are one to one.
\end{remark}
We now characterize the spectrum of the monodromy operator.
\subsection{Proof of proposition~\ref{prop:stab-mtw}}
\begin{proof}
\textbf{Proof of 1.}
	Using \cite{engel_one-parameter_2000}[corollary VI.9.20], we find that
	\[
	\mathbf S = \mT_0(T)+\int_0^T\mT_0(T-s)\tilde\mW(s)\mU(s,0)ds
	\]
	where $\mT_0$ is the semigroup generated by $\bar c\partial -\mathbf L_0$.
	We consider the cutoff functions $\chi_n$ defined previously and define
	\[
	\mK_n:=\int_0^T\mT_0(T-s)\chi_n\tilde\mW(s)\mU(s,0)ds.
	\]
	From the proof of proposition~\ref{prop:linstabTW}, we know that $\chi_n\tilde\mW(s)$ is compact in $\Lp^2$ for all $s$. It follows that $\mT_0(T-s)\chi_n\tilde\mW(s)\mU(s,0)$ is also compact and so is $\mK_n$ as the set of compact operators is closed in $\mathcal L(\Lp^2)$.
	
	We now introduce
	\[
	\mB_n:=\int_0^T\mT_0(T-s)(1-\chi_n)\tilde\mW(s)\mU(s,0)ds
	\]
	whose norm can be made small as $\norm{\mB_n}\to 0$ as $n\to\infty$ thanks to the bounds \eqref{eq:bound-T0} and \eqref{eq:bound-U}. Thus, from lemma~\ref{lem:T0}, $\norm{\mT_0(T)}<1$, we can chose $n$ large enough so that $\norm{\mT_0(T)+\mB_n}<1$ ; the spectrum $\Sigma(\mT_0(T)+\mB_n)$ is strictly in the unit disk.
	We apply corollary~\ref{coro:SigTpK} to $(\mT_0(T)+\mB_n) + \mK_n$ thanks to $\Lp^2$ and $\Hp^1$ being separable Hilbert spaces \cite{brezis_functional_2010}[Theorem 4.13, Proposition 9.1]. It shows that $\Sigma(S(T))\setminus\Sigma(\mT_0(T)+\mB_n)$ is composed of eigenvalues with finite algebraic multiplicities. This concludes the first part of the proposition.
	
	\textbf{Proof of 2.}
	The MTW $\tilde\mv$ satisfies
	\[
	\frac{d}{dt}\tilde \mv = \bar c\partial\tilde\mv+\mF(\tilde\mv;p)
	\]
	which can be differentiated with respect to time if $\tilde\mv\in H^2_{per}([0,T], \Hp^1)$ to show that $\frac{d}{dt}\tilde\mv(t)$ is solution to \eqref{eq:linper} implying that $\frac{d}{dt}\tilde\mv(0)\in\ker\left(\mathbf S-\id\right)$. Similarly, one can show that $\partial\tilde\mv(t)$ is solution to \eqref{eq:linper} so that $\partial\tilde\mv(0)\in\ker\left(\mathbf S-\id\right)$. Assume that the kernel is one dimensional, that is there is $\alpha\neq 0$ such that $\frac{d}{dt}\tilde\mv(0) = \alpha\partial\tilde\mv(0)$. By uniqueness of the solution to \eqref{eq:linper}, we find that $\frac{d}{dt}\tilde\mv = \alpha\partial\tilde\mv$ hence $\tilde\mv(x,t) = h(x+\alpha t)$ for some function $h$ which is $\alpha T$ periodic. This concludes the second part of the proposition.
	
	\textbf{Proof of 3.}
	We note that if $t=n T +\tilde t$ with $0\leq \tilde t<T$, one has 
	\[
		\mU(t,s) = \mU(n T +\tilde t,s) = \mU(\tilde t, s)\mU(nT+s,s) = \mU(\tilde t,s)\mathbf S(s)^n
	\]
	which implies $\norm{\mU(t,s)\mv}_2\leq C\norm{\mathbf S(s)^n\mv}_2$ by \eqref{eq:bound-U}.
	The spectrum of \textit{period map} being independent of $s$, we conclude that $\mv$ does not have any component in $\mathbf S(s)$ generalized eigenspace $\mathbb E_1(s)$ for the eigenvalue $1$. 
%	We denote by $\mathbf P_1$ its spectral projector, hence $(I-\mathbf P_1)\mv = \mv$. 
	From \cite{iooss_bifurcation_1979}~technical lemma, we conclude that $\mathbf S(s)^n\mv\to 0$ and the conclusion follows.
\end{proof}

\section {Numerical methods}\label{sec:nm}
Owing to the number of space variables, we aim at performing the computations in a matrix-free setting freeing us from the limitation on the number of variables. In effect, for a space discretization of size $N$, the number of unknowns is $2N+1$ for a TW.

Note that \textit{pde2path} \cite{uecker_pde2path_2014} supports the computation of waves through the addition of constraints to the model formulation. However, the software lacks shooting capabilities and that is why we use \href{https://github.com/bifurcationkit/BifurcationKit.jl}{BifurcationKit.jl} \cite{veltz_bifurcationkitjl_2020} instead. The implemented methods require very little modification of this Julia package and will be added to the official package in the near future.

\subsection{Traveling waves}
The traveling waves are numerically approximated as zeros of \eqref{eq:wave} with Neumann boundary conditions for $\partial$ using a Krylov-Newton \cite{kelley_iterative_1995, rankin_continuation_2014} algorithm with preconditioner \cite{sanchez_computation_2013} $Pr:=(-I+c_{prec}\partial , -I/\tau+c_{prec}\partial)$ for $c_{prec}$ estimated with direct simulation. It is very cheap to cache the sparse LU decomposition of $Pr$ even for large discretizations. Without this preconditioner, the Krylov-Newton algorithm converges very slowly or does not converge at all. The iterative linear solver used for the Newton iterations is GMRES \cite{barrett_templates_1994}.

\subsection{Stability of traveling waves}
The linear stability of a TW $(\bar\mv,\bar c)$ is deduced from the computation of the spectrum of $\mathbf A/\bar c$ in \eqref{eq:oplinstab}. This is performed using a shift-invert method \cite{saad_numerical_2011} based on the preconditioner $Pr$ in order to extract the right-most eigenvalue. 
However, we found that the method sometimes misses eigenvalues with large imaginary parts stemming from the transport operator $\partial$, so we validate our results with the slower computation of the spectrum of $e^{t\mathbf A}$ (in effect the semigroup) for a fixed and small $t>0$, see \cite{kevrekidis_computational_2020}. In all cases, the iterative eigensolver is based on the Arnoldi algorithm \cite{saad_numerical_2011}.

Another method consists in computing the spectrum of $\mA_\mP$ based on lemma~\ref{lemma:sigmaAp}. In the context of bifurcation theory, this is really advantageous as it removes the $0$ eigenvalue associated with $\partial\bar\mv$ with is nonzero (numerically) although quite small in practice ($\approx 10^{-9}$).

\subsection{Numerical bifurcation analysis of traveling waves}

To study the dependence of the TW on a parameter $p$, we use
a pseudo-arclength continuation \cite{govaerts_numerical_2000} method which allows us to obtain curves of solutions $(\bar\mfu(s),\bar c(s),p(s))$, $s$ being the arclength.
The wave reference $\hat \mfu$ and the wave speed of the preconditioner $Pr$ is updated with the previously found wave solution.

The detection of bifurcation points in a continuation run is performed using the bisection algorithm implemented in \cite{veltz_bifurcationkitjl_2020}. The continuation of codim 1 bifurcation points (Fold / Hopf) in the wave frame is based on a minimally augmented formulation of \eqref{eq:pddae} with the diagonal mass matrix $\mM=diag(I_{2N},0)$. Finally, the codimension 2 bifurcations are computed along the Fold/Hopf curve based on theorem~\ref{thm:CM}.

\subsection{Direct simulations}
The computation of the solutions of the DAE \eqref{eq:pddae} and the ODE \eqref{eq:mtw_ode} is based on the adaptive and implicit time stepper \textrm{Rodas4P2} of the Julia package \href{https://github.com/SciML/DifferentialEquations.jl}{DifferentialEquations.jl} \cite{rackauckas2017differentialequations} using a Jacobian-free formulation with preconditioner $Pr$ updated at time steps for which the iterative solver GMRES has longer iterations than a prescribed number.

\subsection{Modulated traveling waves}\label{section:schemesmtp}

In this paragraph, we provide details for computing MTW using two general different methods, each one having its own pros and cons. We recall that MTW are periodic solutions with period $T$ of \eqref{eq:pddae} which we rewrite
\begin{equation}\tag{MTW-DAE}\label{eq:FDDAE}
	\left\{
	\begin{array}{ccl}
		\frac{d\mv}{dt} &=& \lambda(t)\partial\mv(t) + \mathbf F(\mv(t);p)\\
		0 &=&  \langle\mv(t)-\hat \mfu, \partial\hat \mfu \rangle_{L^2}\\
		0 &=&s(\mv(0),T)
	\end{array}\right.
\end{equation}
where $s(\mv(0),T)$ is a section to remove the time phase invariance. Analogously, the MTW can be found by solving the ODE
\begin{equation}\tag{MTW-ODE}\label{eq:FDODE}
	\left\{
	\begin{array}{ccl}
		\frac{d\mv}{dt} &=& \bar c\partial\mv(t) + \mathbf F(\mv(t);p)\\
		0 &=&s(\mv(0),T).
	\end{array}\right.
\end{equation}
These problems are solved using a time discretization yielding a large spatio-temporal problem or using shooting. The spatio-temporal formulation is the fastest to solve in our experiments but it is not adaptive and fails to capture orbits with large periods or stiff changes. The shooting method is slower but more robust and simpler to implement.

\subsection{Modulated traveling waves, finite differences}\label{section:boxscheme}
In this case, the box scheme  \cite{keller_accurate_1974,uecker_numerical_2021,knobloch_origin_2021}  based on a trapezoidal one is applied to find the solutions $(\mv(t_i), \lambda(t_i), T)$ of \eqref{eq:FDDAE}. More precisely, using a uniform mesh of length $M$ of the time interval $0=t_1<\cdots <t_M=1$, we discretize the first equation of \eqref{eq:FDDAE} as follows:
\begin{equation}\label{eq:boxscheme}
	\frac{T}{h}\left[\mv(t_i)-\mv(t_{i-1})\right] =\frac12\lambda(t_i)\partial\mv(t_i) +\frac12\lambda(t_{i-1})\partial\mv(t_{i-1})
	+\mF\left(\frac12\left[\mv(t_i)+\mv(t_{i-1})\right];p\right),
\end{equation}
where $1\leq i\leq M, h:=\frac{1}{M}$ and $\mv(t_1) := \mv(1)$.
Analogously, we denote by $(\mv(t_i), T, \bar c)$ the discretized solution to \eqref{eq:FDODE}. This gives a number  of unknowns of $3MN+1$ or $2NM+2$ depending on the formulation. In both cases, the solution to the linear problem associated with the Krylov-Newton algorithm is unlikely to converge without a preconditioner. 
When solving \eqref{eq:FDODE}, we use the precondionner $Pr_{ode}(c_{prec})$ where $c_{prec}$ is close to the speed $\bar c$ we are looking for. $Pr_{ode}(c_{prec})$ is defined as the linear operator obtained by replacing $\mF$ in \eqref{eq:boxscheme} by $-\mathbf L_0$ which makes it sparse.
When solving \eqref{eq:FDDAE}, the precondionner $Pr_{dae}(\tilde\mu(t_i))$ is built similarly with $\mF$ replaced by $(-\mathbf L_0,\id)$ where the identity operator is for the $\lambda$ component of \eqref{eq:FDDAE}. We found that these cheap preconditioners allows fast computation of the MTW. In practice, 1 Newton iteration takes roughly 1s for $M\approx 50$.

In this case, the stability of the MTW is assessed by the computation of the Floquet exponents.
Their computation is based on the Arnoldi algorithm \cite{saad_numerical_2011}.

\subsection{Modulated traveling waves, shooting}
The second way to find MTW is to apply a shooting method \cite{garcia_continuation_2016, garcia_comparison_2010} based on a preconditioned adaptive and implicit time stepper. More precisely, to solve \ref{eq:FDDAE}, we look for $(\mv, \lambda, T)$ where $\mv$ is a (several) point(s)  on the orbit. We note $\phi_{dae}^t(\mv,\lambda)$ the flow of \eqref{eq:pddae} and thus seek zeros of 
\begin{equation}\tag{SH-DAE}
	\label{eq:SHDAE}
	\left\{
	\begin{array}{ccl}
		0 &=& (\mv,\lambda)- \phi_{dae}^T(\mv,\lambda)\\
		0 &=&s(\mv,T). 
	\end{array}\right.
\end{equation}
Finally, one last method is to apply a standard shooting to the ODE \eqref{eq:FDODE} to find the zeros $(\mv,\bar c,T)$ of the functional:
\begin{equation}\tag{SH-ODE}
	\label{eq:SHODE}
	\left\{
	\begin{array}{ccl}
		0 &=& \mv- \phi^T(\mv;\bar c)\\
		0 &=&s(\mv,T),
	\end{array}\right.
\end{equation}
where the flow of \eqref{eq:mtw_ode} is written $\phi^t(\mv; \bar c)$.
In both cases,  the linear problem associated with the Krylov-Newton algorithm is solved using the iterative linear solver GMRES where the jacobian operator of the functional is evaluated with centered finite differences.

The computation of the Floquet multipliers associated to \eqref{eq:mtw_ode} is based on the Arnoldi algorithm \cite{saad_numerical_2011} applied to the monodromy operator along the MTW. 

\section{Numerical results}\label{section:simus}
For the applications, we assume the following values for the parameters if not specified otherwise: 
$\beta = 50.0,
a_{ee} = 1.0, a_{ei} = 1.0, 
a_{ie} = 1.3, a_{ii} = 0.25,
\theta_e = 0.12, \theta_i = 0.3, \sigma_e =  \sigma_i=10, \tau = 0.82$. We note that there are 3 equilibrium points for the clamped system for these parameters. We also assume that the kernels are gaussian:
\[
K_{p}(x):=\frac{1}{\sigma_p\sqrt{2\pi} } \exp \left(-\frac{x^2}{2\sigma_p^2}\right).
\]
We apply the tools developed above for the study of \eqref{eq:nfe}. We first simulate \eqref{eq:nfe} and \eqref{eq:pddae} in Figure~\ref{fig0}  to show the kind of TPs we are interested in. In the top row is featured a stable TP which is destabilized by decreasing the parameter $\tau$ (bottom row) into a \textit{modulated traveling pulse} (MTP). 

We then perform numerical continuation of the TP with respect to $\tau$ and find two Hopf bifurcations. This allows to perform continuation (see Figure~\ref{fig:codim2}) of the Hopf points in the plane $(\theta_i,\tau)$ where a Bogdanov-Takens (BT) bifurcation is detected at $(\theta_i,\tau)\approx (0.318583, 1.953818)$. This BT point has been refined using a minimally augmented formulation \cite{govaerts_computation_1993}. A curve a Fold points is then computed starting from the BT point. The branch of homoclinics orbits from the BT point has not been computed due to the lack of associated method in \href{https://github.com/bifurcationkit/BifurcationKit.jl}{BifurcationKit.jl} version 0.2. The Hopf branch seems to converge to a limit point where the TP converges to a traveling front. 

\begin{figure}[h!]
	\centering
	\includegraphics[width=0.8\textwidth]{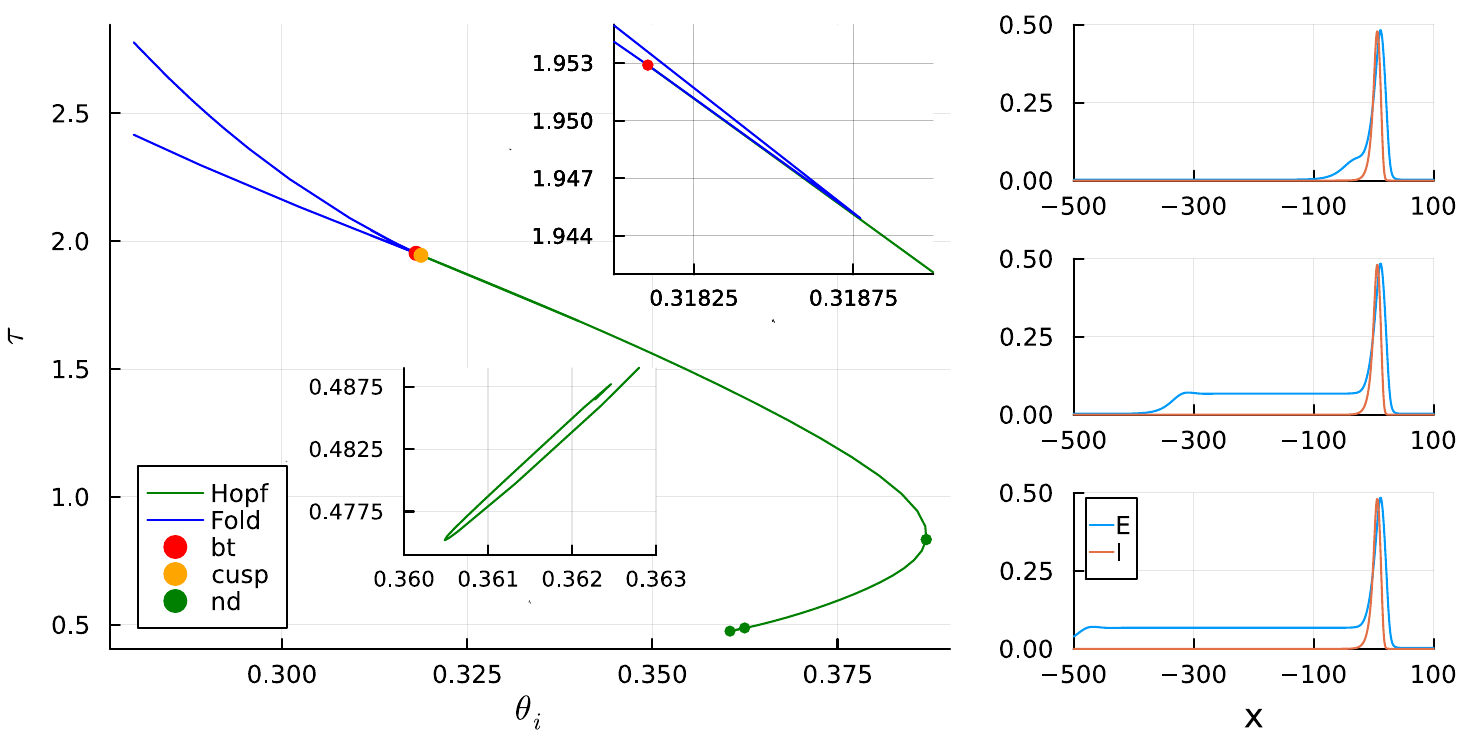}
	\caption{Left: continuation of Hopf/Folds of TPs in the plane $(\theta_i,\tau)$. The insets show a zoom near the BT point and the limit point on the Hopf branch. Legend: BT means Bogdanov-Takens, ND means Fold. Right: TP solutions near the limit point on the Hopf branch.}
		\label{fig:codim2}
\end{figure}

We next consider a Hopf bifurcation for a given $\theta_i$ and computes the curve of MTPs emanating from it as function of the parameter $\tau$. It seems numerically that there are two main regimes. When the parameters are close to the Fold of Hopf points (at $(\theta_i,\tau)\approx (0.386975,0.833440)$ in Figure~\ref{fig:codim2}), the branch of (stable) MTPs connects the two Hopf bifurcation points. This case is shown in Figure~\ref{fig:mtp1}.

\begin{figure}[ht!]
	\centering
	\includegraphics[width=0.9\textwidth]{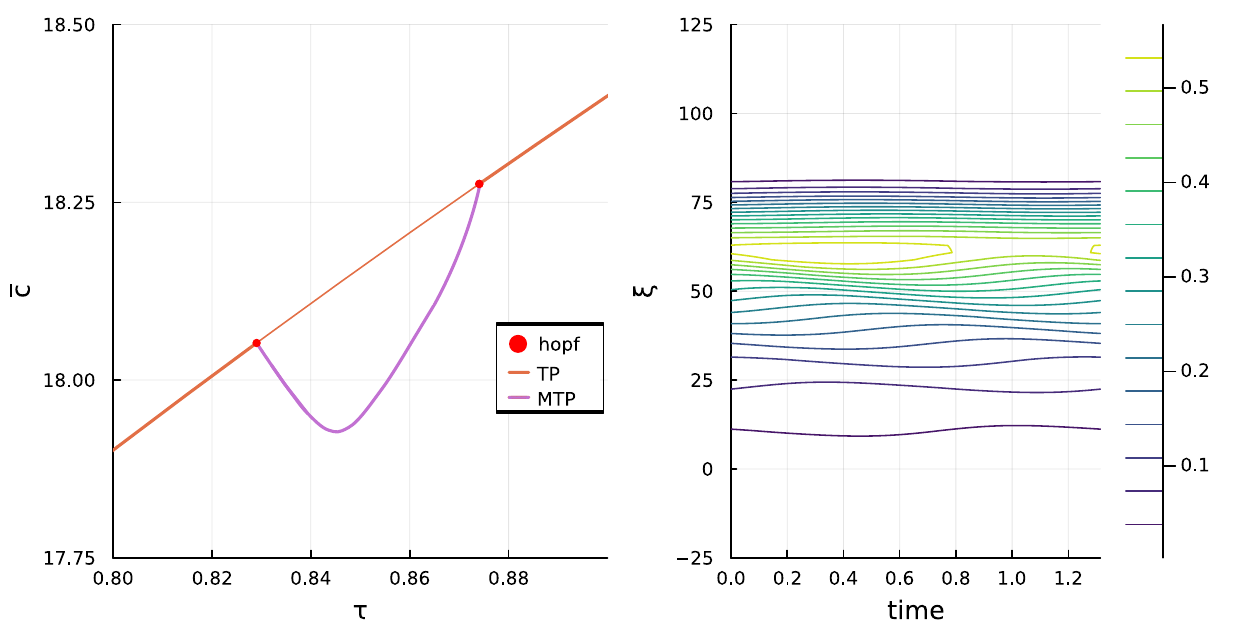}
	\caption{Left: curves of TPs and MTPs for $\theta_i = 0.3869$.  Right: example of MTP in wave frame for $\tau \approx 0.8268$. The branch of MTP was computed with a single standard shooting \eqref{eq:SHODE}. For the branch of MTPs, we plot the value of the speed $\bar c$ as expressed in \eqref{eq:SHODE}. The stable states are marked with a thick line and the unstable ones with a thin line.}
		\label{fig:mtp1}
\end{figure}

If we decrease a bit the threshold to $\theta_i=0.38$ as in Figure~\ref{fig:mtp2}, the connection between the two Hopf bifurcation points by the branch of MTPs opens up and it seems numerically that the branches of MTPs converges to limit points $LP_1, LP_2$. The period of the MTPs seems to converge as one approaches the points $LP_1$ whereas it seems to grow unbounded near $LP_2$ (the continuation failed near $LP_2$). The branch of MTPs on the left part of Figure~\ref{fig:mtp2} seems to converge to a modulated traveling front which connects two different values unlike the TP which connects two equal (small) values ; this is analyzed more in details here after.  The branch of MTPs on the right part of seems to converge to an orbit composed of two waves with different speeds as can be seen from the MTP profiles in wave frame on the right column of Figure~\ref{fig:mtp2}.

\begin{figure}[h!]
	\centering
	\includegraphics[width=0.9\textwidth]{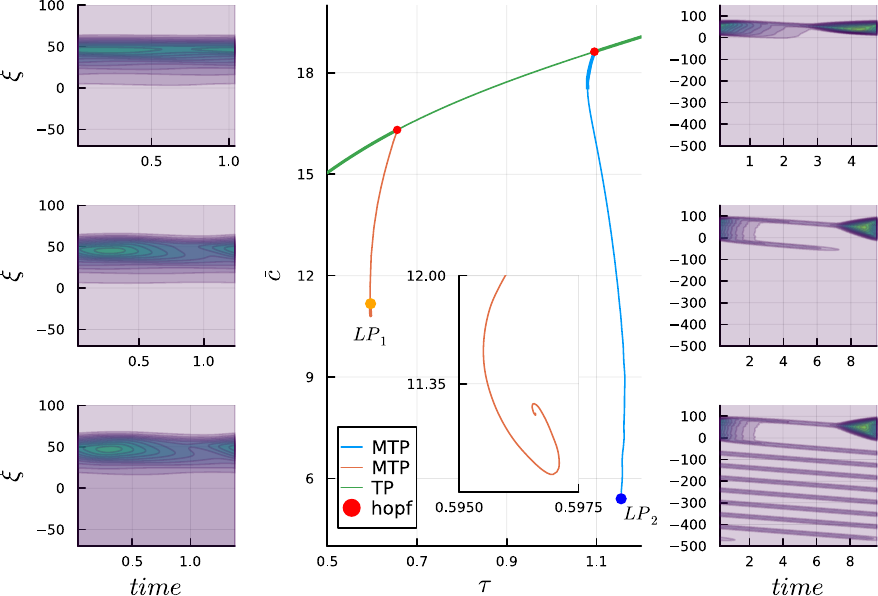}
	\caption{Middle: curves of TPs and MTPs for $\theta_i = 0.38$. Inset: zoom on the left MTP branch.  Left: example of MTPs on the left branch in the wave frame. Right: example of MTPs on the right branch in the wave frame. The MTP branches were computed with a single standard shooting \eqref{eq:SHODE}. For the branch of MTPs, we plot the value of the speed $\bar c$ as expressed in \eqref{eq:SHODE}. The stable states are marked with a thick line and the unstable ones with a thin line. The points $LP_1, LP_2$ marked with orange/blue dots represents limits points at which the continuation stopped.}
		\label{fig:mtp2}
\end{figure}
In Figure~\ref{fig:snakingL}, we show the profiles of the MTPs as they approach the point $LP_1$. This is an evidence of snaking behavior as is commonly found near homoclinic points. On Figure~\ref{fig:snakingL}~Right, an example of MTP is shown in wave frame which hints at how the tail is developing along the MTP branch, supporting the fact that the MTP seems to converge to a modulated traveling front.

\begin{figure}[htbp!]
	\centering
	\includegraphics[width=0.99\textwidth]{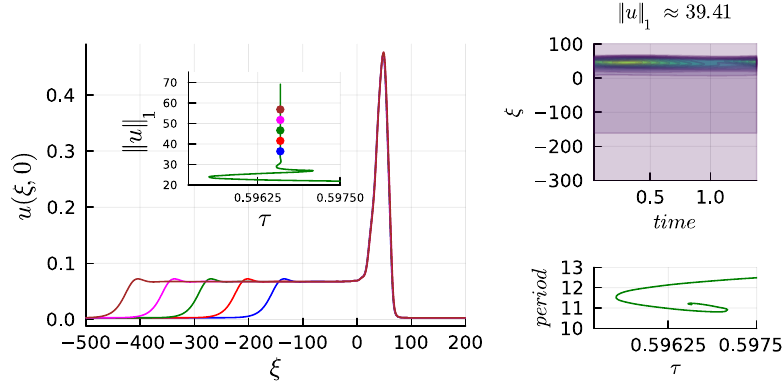}
	\caption{Snaking of MTPs for $\theta_i = 0.38$, left branch in Figure~\ref{fig:mtp2} near $HP_1$. Left: plot of one section for several MTPs near the point $HP_1$. The inset shows a zoom of the MTP branch near the point $HP_1$. The dots are Fold bifurcations points. Right: MTP profile for the blue mark on the inset figure. The $L^1$ norm is defined as $\norm{u}_1:=\frac{1}{T}\int_{0}^{T}\int_{\Omega}\ \lvert u(\xi,t)\rvert\ d\xi dt$.}
		\label{fig:snakingL}
\end{figure}

\section{Discussion}\label{section:discussion}
In this work, we focused on the theoretical / numerical analysis of the dynamics of TW and MTW in a two populations neural field model with one dimensional cortex. We proved a principle of linearised stability for TW, develop a bifurcation theory of TW and characterize the spectral properties associated to MTW. We based our analysis on the freezing method yielding a DAE which allows an efficient description of the dynamics. The only point left for future study is the proof of a principle of linearized stability for MTW. Compared to reaction-diffusion systems, the analysis of MTW for neural field  equations seems more challenging as the dynamics in the wave frame is not regularizing, \textit{i.e.} the linearized operator is not sectorial. This prevents us from defining a Poincar\'e return map for example.

Then, we presented numerical methods to compute MTW. This is where the freezing method really shines as it allows to simulate the wave profile for as long as we want, we are not limited by the simulation domain as in \cite{harris_traveling_2018} for example. We found these MTW from Hopf bifurcations of TP.
Interestingly, we numerically found Fold, Hopf and Bogdanov-Takens bifurcations of TP which were analytically studied in the case of a Heaviside nonlinearity in the very detailed work \cite{folias_traveling_2017}. The study of stationary traveling waves for different ``mean field'' of neural populations has recently been performed in \cite{laing_moving_2020,palkar_inhibitory_2022, byrne_mean-field_2022,laing_periodic_2023}. We are then able to follow the MTW as function of a parameter (in this case the time scale ratio $\tau$) and showed evidences for snaking of MTP.

We focused on single pulses and while our stability results do not depend on their specific shape, dedicated tools have been developed to  study the existence of multi-pulse solutions. This would be an interesting area of future research in the context of MTPs.

We discarded the presence of conduction delays in the model albeit they seem to be very important in the modeling of visual cortex \cite{muller_stimulus-evoked_2014}. In this context, our analysis does not hold and a lot remains to be done and a dedicated framework seems to be lacking in this case (see \cite{meijer_travelling_2014} for a numerical study in the context of neural fields).

The extension of the theory to two-dimensional domains is not completely direct and we already expect the loss of exponential stability of TP based on \cite{faye_multidimensional_2015}. The study of circular pulses seems of intermediate difficulty and of great interest given their relevance to visual cortex \cite{muller_stimulus-evoked_2014}.

Finally, the interplay between noise and wave dynamics is a fascinating topic \cite{kilpatrick_pulse_2014,kilpatrick_stochastic_2015, inglis_general_2016,lang_finite-size_2017}. A lot remains to be done in this context especially for MTW and 2d cortices.

\section{Acknowledgments}
R. Veltz thanks the support of the ANR project ChaMaNe (ANR-19-CE40-0024, CHAllenges in MAthematical
NEuroscience).

This work was also supported by the European Union’s Horizon 2020 Framework Programme for Research and Innovation under the Specific Grant Agreement No. 945539 (Human Brain Project SGA3). The Human Brain Project is a collaborative, interdisciplinary effort including groups from more than 20 countries.

\appendix
\section{Technical details}

All numerical computations were performed in the Julia programming language (version 1.8.2). Numerical bifurcation analysis was performed using \href{https://github.com/bifurcationkit/BifurcationKit.jl}{BifurcationKit.jl} (version 0.2.1).
Unless otherwise specified, all computations were made using $N = 2048$ points on $[-L,L]$ with $L =500$. We could have used $N\sim 10^6$ on a laptop for the analysis of TPs but the computation of MTPs would have required the use of dedicated hardware, \textit{e.g.} GPU and we did not find that increasing $N$ dramatically changed our numerical observations. Centered finite differences of order 2 were used for the spatial derivative $\partial$. The newton tolerance for all computation was set to $1e{-9}$ in supremum norm. The numerically stable computation of Floquet exponents using box scheme (section~\ref{section:boxscheme}) proved to be brittle. We therefore refrain from computing bifurcations from MTPs using finite differences, and instead check their stability using the shooting method.

\section{Proof of lemma~\ref{lemma:FCk}}\label{proof:lemma47}
\rem{Brezis. 
	
$\Hp^1$ est bien defini. Convolution $\rho\in L^1, v\in W^{1,p} $ alors $\rho\star v\in W^{1p}$, 
	
$L^1*L^p\subset L^p$, $L^p*L^q\subset L^r$

$L^1*W^{1,p}\subset W^{1,p}$ on page 211}
\textbf{Case $r=1$.}
We give the proof in the case of a single component and we assume that $\tau=1$ to simplify notations.
We first show that $ \tilde\mF(\mv;p)\in\Lp^2$ for $\mv\in\Lp^2$.	Because, $S$ is globally Lipschitz with constant $C_s$, one finds
\[
\forall x\in\R,\quad	\lvert S\left(\mW\cdot\bar\mv+\mW\cdot\mv\right)-S\left(\mW\cdot\bar\mv\right)\rvert(x) \leq C_s \lvert \mW\cdot\mv(x) \rvert.
\]
Since $\mW\in \mathrm L^1(\R, \mathcal M_2(\R))$, Young's inequality for convolution \cite{brezis_functional_2010} gives $\norm{\mW\cdot\mv}_2\leq \norm{\mW}_1\norm{\mv}_2$ and then (see \eqref{eq:tildeF})
\[
\norm{\tilde\mF(\mv)}_2\leq C\norm{\mW}_1\norm{\mv}_2
\]
for some $C>0$.
Using a Taylor expansion of $S$ at order 1, one has for $\mv,\mh\in\Lp^2$:
\begin{multline}\label{eq:diff-O2}
	S(\mW\cdot(\bar\mv+\mv+\mh)) = S(\mW\cdot(\bar\mv+\mv)) + S'\left(\mW\cdot(\bar\mv+\mv)\right)\mW\cdot\mh \\ + \frac12\int_0^1(1-s)S^{(2)}\left(\mW\cdot\bar\mv+\mW\cdot\mv+s\mW\cdot\mh\right)\cdot\left(\mW\cdot\mh\right)^2ds.
\end{multline}
This provides a candidate for the differential of $\tilde\mF$ at $\mv$:
\[
d\tilde\mF(\mv;p)\cdot\mh = -\mathbf L_0\mh + S'\left(\mW\cdot(\bar\mv+\mv)\right)\mW\cdot\mh.
\]
$d\tilde\mF(\mv;p)$ is a continuous operator $d\tilde\mF(\mv;p)\in\mathcal L(\Lp^2)$ using the boundedness of $S$. It remains to show that the reminder in \eqref{eq:diff-O2} is $o(\mh)$. Using the hypothesis $\mW\in \mathrm \Wp(\R,\mathcal M_2(\R))$, we find that $\mW\cdot\mh\in W^{1,2}$ using Young's inequality for convolution and
\[
\norm{\mW\cdot\mh}_{W^{1,2}}\leq \norm{\mW}_{\mathrm W^{1,1}}\norm{\mh}_2.
\]
$\rm H^1 =\rm W^{1,2}$ being a Banach algebra, there is $C>0$ such that $\norm{\mfu^2}_{W^{1,2}}\leq C\norm{\mfu}^2_{W^{1,2}}$ for all $\mfu\in W^{1,2}$. This implies that
\[
\norm{(\mW\cdot\mh)^2}_2\leq 
\norm{(\mW\cdot\mh)^2}_{\Hp^1}\leq 
C\norm{\mW\cdot\mh}^2_{\Hp^1}\leq 
C \norm{\mW}_{\mathrm W^{1,1}}^2\norm{\mh}_2^2.
\]
Using the boundedness of $S^{(2)}$, this implies that the reminder is $o(\mh)$ in $\Lp^2$. This concludes the proof that $\tilde\mF\in C^1(\Lp^2,\Lp^2)$.

\textbf{General $r$ case.} It is then straightforward to conclude that $\tilde\mF\in C^r(\Lp^2,\Lp^2)$ using the same method and relying on the fact that $\rm H^1$ is a Banach algebra.
\section{Functional analysis}
We quote the result  \cite{reed_functional_1980}[Theorem VI.14].
\begin{theorem}[Analytic Fredholm theorem]
	Let $D$ be an open connected subset of $\mathbb C$ and $\mathcal H$ a separable Hilbert space. Let $\mK: D \to\mathcal L(\mathcal H)$ be an analytic operator-valued function such that $\mK(z)$ is compact for each $z\in D$. Then, either
	\begin{itemize}
		\item $(I-\mK(z))^{-1}$ exists for no $z\in D$
		\item $(I-\mK(z))^{-1}$ exists for all $z\in D\setminus S$ where $S$ is a discrete subset of $D$
		(\textit{i.e.} a set which has no limit points in $D$). In this case, $(I-\mK(z))^{-1}$ is meromorphic in $D$, analytic in $D\setminus S$, the residues at the poles are finite rank operators, and if $z\in S$ then $\mK(z)\mv=\mv$ has a nonzero solution in $\mathcal H$.
	\end{itemize}
\end{theorem}

\begin{corollary}\label{coro:SigTpK}
	Let $\mA = \mA_0+\mB\in \mathcal L(\mathcal H)$ be a bounded linear operator in a separable Hilbert space $\mathcal H$ with $\mA_0\in\mathcal L(\mathcal H)$ and $\mB\in\mathcal L(\mathcal H)$ compact. Then $\Sigma(\mA)\setminus\Sigma(\mA_0)$ is composed of a discrete set of eigenvalues with finite algebraic multiplicity (\textit{i.e.} the corresponding generalized eigenspace is finite dimensional).
\end{corollary}
\begin{proof}
	Let $\mK(z) = (zI-\mA_0)^{-1}\mB$. Then $\mK(z)$ is an analytical compact operator valued function on the resolvent set $D=\rho(\mA_0)$. 	From $\lim\limits_{z\to\infty}\mK(z)=I$, the first alternative of the analytic Fredholm theorem is not verified. Thus, there is a discrete set $S$ so that $I-\mK(z)$ is invertible for $z\in D\setminus S$ and if $z\in S$, $\mK(z)\mv=\mv$ has a non zero solution in $\mathcal H$. Thus $\mv$ is an eigenvector of $\mA$ for the eigenvalue $z$. The fact that such eigenvalues $z$ have finite multiplicity follows
	immediately from the compactness of $\mK(z)$.
\end{proof}

\begin{lemma}\label{lem-pv}
	For $v\in \rm H^1$, then $\int \mv\partial \mv = 0$.
\end{lemma}
\begin{proof}
	For simplicity, we only do the case where $\mv(x)$ has a single complex component.
	This is just $\int\partial(\mv^2)=[\mv^2]_{-\infty}^\infty=0$.
\end{proof}

%\appendix
%\section{Example Appendix Section}
%\label{app1}
%%% For citations use: 
%%%       \cite{<label>} ==> [1]
%
%%%
%Example citation, See \cite{lamport94}.




%\begin{thebibliography}{00}
%
%%% For numbered reference style
%%% Text of bibliographic item
%
%\bibitem{lamport94}
%  \textit{\LaTeX: a document preparation system},
%  Addison Wesley, Massachusetts,
%\end{thebibliography}
%\bibliographystyle{elsarticle-num} 
\bibliographystyle{alphadin}
\bibliography{references}
\end{document}